\documentclass[11pt]{amsart}
\usepackage{amsfonts}
\usepackage{amssymb}
\input xy
\xyoption{all}

\newtheorem{theorem} {Theorem}
\newtheorem{lemma}[theorem] {Lemma}

\newtheorem{corollary} [theorem]{Corollary}
\newtheorem{proposition}[theorem] {Proposition}
\newtheorem{theorem2} [theorem]{Theorem}

\newtheorem{corollary2} [theorem]{Corollary}
\newtheorem{remark} [theorem]{Remark}
\theoremstyle{definition}

\theoremstyle{definition}
\newtheorem{example} [theorem]{Example}

\begin{document}
\title[varieties of complexes]{Varieties of complexes of fixed rank}
\author{Darmajid}
\author{Bernt Tore Jensen}

\begin{abstract}
We study varieties of complexes of projective modules with fixed ranks, and
relate these varieties to the varieties of their homologies. We show that
for an algebra of global dimension at most two, these two varieties are related by a
pair of morphisms which are smooth with irreducible fibres.
\end{abstract}
\maketitle

\section*{Introduction}

Let $\Bbbk$ be an algebraically closed field.
For a finite dimensional $\Bbbk$-algebra $A$, Huisgen-Zimmermann and
Saorin \cite{SH-Z01} define an affine
variety which parameterizes bounded complexes of $A$-modules.
Jensen, Su, and Zimmermann \cite{JSZ05} use varieties of
complexes of projective modules to study degeneration in derived categories.
These varieties, denoted by $comproj_{\mathbf{d}}^{A}$, are
defined as the affine variety of all differentials
$( \partial _{i})_{i\in \mathbb{Z}}$
for a fixed choice of projective modules with multiplicities of indecomposable projective
summands encoded
by a sequence of vectors $\mathbf{d}=(\mathbf{d}_i)_{i\in \mathbb{Z}}$,
called a dimension array.
In \cite{JS05} Jensen and Su use these
to show that there is a well-defined notion of type of singularity in the derived
category of a finite dimensional algebra, and prove that this type coincides
with the type of singularity for the homology for hereditary algebras.

The purpose of this paper is to study geometric properties of
varieties of complexes of projective modules with a fixed rank and to
relate properties of these varieties to varieties of
representations via the homology functor.

Let $\Lambda$ denote a $\Bbbk$-algebra of global dimension at most two. We define $comproj_{\mathbf{d},\mathbf{r}}^{\Lambda}$
to be subset of $comproj_{\mathbf{d}}^{\Lambda}$ where the
differential $\partial=(\partial_i)_{i\in \mathbb{Z}}$ has fixed ranks
encoded in the sequence of dimension vectors $\mathbf{r}$. Rank
is an lower semi-continuous function on $comproj_{\mathbf{d}}^{\Lambda}$,
and so we have a disjoint union of locally closed subsets
$$comproj_{\mathbf{d}}^{\Lambda}=
\bigcup_{\mathbf{r}}comproj_{\mathbf{d},\mathbf{r}}^{\Lambda}.$$

We construct a variety $comhom_{\mathbf{d},\mathbf{r}}^{\Lambda}$ whose points
are complexes together with their homology, and two morphisms $\pi $ and $\rho $
$$\xymatrix{& comhom_{\mathbf{d},\mathbf{r}}^{\Lambda} \ar[dl]_\pi \ar[dr]^\rho & \\ comproj_{{\mathbf{d}},{\mathbf{r}}}^{\Lambda}
&& rep_{\mathbf{h}}^{\Lambda}}$$
such that $\rho \left( \pi ^{-1}\left(X\right) \right)$ is in the orbit of
the homology $H^{\ast }\left(
X\right) $ for all $X\in comproj_{\mathbf{d},\mathbf{r}}^{\Lambda}$ and $rep_{\mathbf{h}}^{\Lambda}$
is the variety of $\mathbb{Z}$-graded $\Lambda$-modules with dimension vectors given by the sequence $\mathbf{h}$.
Our main result is the following

\begin{theorem2}
The morphisms $\pi$ and $\rho $ are smooth with irreducible rational fibres.
\end{theorem2}

We prove the theorem by showing that $\rho$ and $\pi$ are compositions
of open immersions, vector bundles and $G$-bundles for irreducible algebraic
groups $G$.

We have an application of our result.

\begin{corollary2}
There is a bijection between
the set of irreducible components of $comproj_{{\mathbf{d}},{\mathbf{r}}}^{\Lambda}$ and the irreducible components of
$im(\rho)$.
\end{corollary2}

We also have the following special case.

\begin{corollary2}
If $\Lambda$ is hereditary, then $comproj_{\mathbf{d},\mathbf{r}}^{\Lambda}$ is irreducible, smooth and rational.
\end{corollary2}

The remainder of this paper is organized as follows. In Section 1 we recall basic
definitions on representations of quivers and the definition of
the variety $comproj_{\mathbf{d}}^{\Lambda}$. In Section 2 we give the definition
of $comhom_{\mathbf{d},\mathbf{r}}^{\Lambda}$ and the morphisms $\pi$ and $\rho$.
The smoothness of $\pi$ and $\rho$ are proved in Section 3, and in Section 4
we discuss some applications and examples.

%=======================================

\vspace{1.5cc}

\section{Preliminaries}

In this section we recall some facts on varieties of representation of quivers and
varieties of complexes. For details, we refer to \cite{ASS06},
\cite{R86}, \cite{JS05} and \cite{JSZ05}. Note that the assumption algebra on global
dimension at most two is not needed in the sequel of this section.

\subsection{Representations of Quivers}

A quiver $Q$ consist of a set of vertices $Q_0$, a set of arrows $Q_1$
and two maps $s,e:Q_{1}\rightarrow Q_{0}$ which assign
to each arrow $\alpha \in Q_{1}$ its source and end vertex, respectively.
The path algebra $\Bbbk Q$ has basis equal to the set of paths in
$Q$, and multiplication of two paths
$\beta$ and $\alpha$ is the composed path
$\beta \alpha$, if $\alpha$ ends where $\beta$ starts, and zero otherwise.
Any finite dimensional algebra $\Lambda$ is Morita equivalent to an algebra
$\Bbbk Q/\mathcal{I}$ for an admissible ideal $\mathcal{I}$. For simplicity we
assume that $\Lambda=\Bbbk Q/\mathcal{I}$.

A representation $V=\left( V_{a},V_{\alpha }\right)_{a\in Q_0,\alpha\in Q_1} $ of a quiver $Q$
consists of a $Q_0$-graded vector space, i.e. a family of vector spaces $V_{a}$ indexes by the vertices $a\in
Q_{0}$, together with a family of linear maps $V_{\alpha }:V_{s\left( \alpha
\right) }\rightarrow V_{e\left( \alpha \right) }$ indexed by the arrows
$\alpha \in Q_{1}$. The dimension vector $dim\left( V\right)\in \mathbb{N}^{Q_0}$ of $V$ is
the vector with components $dim_{\Bbbk }\left( V_a\right)$.
A representation of $\left( Q,\mathcal{I}\right) $ is a
representation with maps that satisfy the relations $\mathcal{I}$.

A homomorphism $\varphi :V\rightarrow W$ between two representations
$V$ and $W$ is a family of
$\Bbbk$-linear maps $\left( \varphi _{a}:V_{a}\rightarrow W_{a}\right) _{a\in
Q_{0}}$ such that for any arrow $\alpha :a\rightarrow b$ the equality $
W_{\alpha }\circ \varphi _{a}=\varphi _{b}\circ V_{\alpha }$ holds.
The vector space of homomorphisms, denoted by $Hom_{\Lambda}(V,W)$ is therefore
a subspace of the space of $Q_0$-graded maps $Hom(V,W)=\prod_{a\in Q_0}Hom(V_a,W_a)$. By the rank $rank(f)$ of
a homomorphism we mean the dimension vector of
the image $im(f)$. The kernel of $f$ is denoted by $ker(f)$.

The category of representations
is equivalent to the category of finite-dimensional left
$\Lambda$-modules. We use this equivalence to
identify modules with representations, and vice versa.

Given a dimension vector $\mathbf{d}=(d_a)_{a\in Q_0}$, we denote by
$rep_{\mathbf{d}}^{\Lambda}$ the affine variety of
representation $\left( Q,\mathcal{I}\right) $ with
$V_{a}=\Bbbk ^{d_{a}}$ for all $a\in
Q_{0} $. Using the standard basis of $\Bbbk ^{d_{a}}$,
any representation in
$rep_{\mathbf{d}}^{\Lambda}$ is given by a tuple of matrices.
The group $$Gl_\mathbf{d}=\prod_{a\in Q_0}Gl_{d_a}$$ acts on
$rep_{\mathbf{d}}^{\Lambda}$ by conjugation such that orbits correspond
to isomorphism classes of representations with dimension vector $\mathbf{d}$.
We denote the $Q_0$-graded
space $\oplus_{a\in Q_0}\Bbbk^{d_a}$ by $\Bbbk^\mathbf{d}$.

\subsection{The Variety of Complexes}

Let $\{P_a|a\in Q_0\}$ be a complete set of representatives of projective
indecomposable $\Lambda$-modules, one from each isomorphism class.
Given a dimension vector $\mathbf{d}=(d_a)_{a\in Q_0}$,
let $$P^\mathbf{d}=\bigoplus
\limits_{a\in Q_0}P_{a}^{d_{a}}.$$  Let $\Theta$ be the Cartan matrix
of $\Lambda$ with rows and columns indexed by $Q_0$. That is,
$\Theta$ is defined as the matrix with
$$\Theta \mathbf{d}=dim(P^\mathbf{d}).$$

Following \cite{JS05}, for every sequence of dimension
vectors $\mathbf{d}:\mathbb{Z}\rightarrow \mathbb{N
}_{0}^{Q_0}$ for which
$\mathbf{d}_{i}=\left( 0,0,\ldots ,0\right) $ for all $i\gg0$ and $i\ll0$,
$comproj_{\mathbf{d}}^{\Lambda}$ is the affine subvariety of
\begin{equation*}
\prod\limits_{i\in \mathbb{Z}}Hom_{\Lambda}(P^{\mathbf{d}
_{i}},P^{\mathbf{d}_{i-1}})
\end{equation*}%
consisting of sequences of maps $\left( \partial _{i}:P^{\mathbf{d}%
_{i}}\rightarrow P^{\mathbf{d}_{i-1}}\right) _{i\in \mathbb{Z}}$
such that $\partial _{i}\partial _{i+1}=0$ for all $i\in \mathbb{Z}$.
Clearly, $comproj_{\mathbf{d}}^{\Lambda}$ parameterizes complexes of
projective $\Lambda$-modules with fixed dimension vectors in each degree.
The sequence $\mathbf{d}$ is called a bounded dimension array.
If the sequence $\mathbf{d}$ is not bounded to the left, then
$comproj_{\mathbf{d}}^{\Lambda}$ is defined as a limit of affine varieties
of truncated complexes \cite{JS05}. If we do not explicitly say so,
any dimension array $\mathbf{d}$ is assumed to be bounded.

The group
\begin{equation*}
G_{\mathbf{d}}:=\prod\limits_{i\in \mathbb{Z}}Aut_{\Lambda}
P^{\mathbf{d}_{i}}
\end{equation*}
acts on $comproj_{\mathbf{d}}^{\Lambda}$ by conjugation and two complexes
in $comproj_{\mathbf{d}}^{\Lambda}$ are in the same orbit if and only if
they are quasi-isomorphic (\cite{JSZ05}, Lemma 1). Moreover, for any two
complexes $M$ and $N$ there exists a (not necessarily bounded) dimension array $\mathbf{d}$
and complexes $M^{\prime },N^{\prime }\in comproj_{\mathbf{d}}^{\Lambda}$ such
that $M$ and $N$ are quasi-isomorphic to $M^{\prime }$ and $N^{\prime }$,
respectively (\cite{JSZ05}, Lemma 2).

For a sequence $\mathbf{r}:\mathbb{Z}\rightarrow \mathbb{N}_{0}^{Q_{0}}$ we
define $comproj_{\mathbf{d},\mathbf{r}}^{\Lambda}$ to be the locally closed $G_{\mathbf{d}}$-stable
subvariety of $comproj_{\mathbf{d}}^{\Lambda}$ consisting of complexes $(\partial_i)_{i\in \mathbb{Z}}$ such that $rank\left( \partial _{i}\right)
={\mathbf r}_{i}$ for each $i\in \mathbb{Z}$.

\section{The Variety $comhom_{\mathbf{d},\mathbf{r}}^{\Lambda}$}

In this section we will construct a variety $comhom_{\mathbf{d},\mathbf{r}
}^{\Lambda }$ consisting of complexes together with their homology. Recall
that $\Lambda$ is assumed to be an algebra of global dimension at most 2.

Let $\left( \partial _{i}\right) _{i\in \mathbb{Z}}\in comproj_{{\mathbf{d}}%
, {\mathbf{r}}}^{\Lambda }$. We define two sequences
\begin{equation*}
\mathbf{k},\mathbf{h}:\mathbb{Z}\rightarrow \mathbb{N}_{0}^{Q_{0}}
\end{equation*}%
with ${\mathbf{k}}_{i}$ equal the dimension vector of the kernel of $%
\partial _{i-1}$ and $\mathbf{h}_{i}$ the dimension vector of homology in
degree $i-1$. That is, ${\mathbf{k}}_{i}=\Theta{\mathbf{d}}_{i-1}-{\mathbf{r}%
}_{i-1}$ and ${\mathbf{h}}_{i}={\mathbf{k}}_{i}-{\mathbf{r}}_{i}$, computed
component-wise.

\begin{lemma}
For any $\left( \partial _{i}\right) _{i\in \mathbb{Z}} \in comproj_{{%
\mathbf{d}},{\mathbf{r}}}^{\Lambda }$, the representation $ker\left(
\partial _{i}\right)$ is a projective representation.
\end{lemma}

\begin{proof}
For any $i$, the sequence
\begin{equation*}
0\rightarrow ker\left( \partial _{i}\right)\rightarrow P^{\mathbf{d}%
_{i}}\rightarrow P^{\mathbf{d}_{i-1}}\rightarrow \frac{P^{\mathbf{d}_{i-1}}}{%
im\left( \partial _{i}\right) }\rightarrow 0
\end{equation*}
is the start of a projective resolution of $\frac{{P}^{\mathbf{d} _{i-1}}}{%
im\left( \partial _{i}\right) }$, and so $ker\left( \partial _{i}\right)$ is
projective, since $\Lambda $ has global dimension at most $2$.
\end{proof}

The lemma allows us to fix projective representations $M_{i}\in rep_{\mathbf{k}_{i}}^{\Lambda }$
such that $M_{i}\cong ker\left( \partial _{i-1}\right)$
for any $\left( \partial _{i}\right)\in comproj_{{\mathbf{d}},{\mathbf{r}}}^{\Lambda}$. Hence, there are
$\Lambda$-monomorphisms $\eta _{i}:M_{i}\rightarrow P^{\mathbf{d}_{i-1}}$,
$\Lambda$-homomorphisms $\phi _{i}:P^{\mathbf{d}_{i}}\rightarrow M_{i}$, and
epimorphisms $\gamma _{i}:M_{i}\rightarrow \Bbbk ^{\mathbf{h}_{i}}$ such
that $\partial _{i}=\eta _{i}\phi _{i}$, $\partial _{i-1}\eta _{i}=0$, and
$im\left( \phi _{i}\right) =ker\left( \gamma _{i}\right) $ for any $i\in \mathbb{Z}$.

\begin{lemma}
\label{comhom}Let $\left( \partial _{i}\right) _{i\in \mathbb{Z}}\in
comproj_{{\mathbf{d}},{\mathbf{r}}}^{\Lambda }$. Then,
\begin{equation*}
rank\left( \phi _{i}\right) =\mathbf{r}_{i}
\end{equation*}%
and there are unique representations $H_{i}\in rep_{\mathbf{h}_{i}}^{\Lambda }$ such that
$\gamma _{i}:M_{i}\rightarrow H_i$ are $\Lambda $-epimorphisms
and $H_{i}\cong
ker\left( \partial _{i-1}\right) /im\left( \partial _{i}\right) $ for any $%
i\in \mathbb{Z}$.
\end{lemma}

\begin{proof}
The condition $\partial _{i}=\eta _{i}\phi _{i}$ means that $ker\left( \phi
_{i}\right) \subseteq ker\left( \partial _{i}\right) $ and the injectivity
of $\eta _{i}$ ensures that $ker\left( \partial _{i}\right) \subseteq
ker\left( \phi _{i}\right) $. Thus, $ker\left( \phi _{i}\right) =ker\left(
\partial _{i}\right) $. By the isomorphism theorem, we have%
\begin{equation*}
im\left( \phi _{i}\right) \cong \frac{P^{\mathbf{d}_{i}}}{ker\left( \phi
_{i}\right) }=\frac{P^{\mathbf{d}_{i}}}{ker\left( \partial _{i}\right) }%
\cong im\left( \partial _{i}\right) .
\end{equation*}%
Hence,%
\begin{equation*}
rank\left( \phi _{i}\right) =rank\left( \partial _{i}\right) =\mathbf{r}_{i}.
\end{equation*}

Now, let $\left( \left( H_{i}\right) _{\alpha }\right) _{\alpha \in Q_{1}}$
be the structure of a representation on $\Bbbk ^{\mathbf{h}_{i}}$. Since $\gamma
_{i}:M_{i}\rightarrow \Bbbk ^{\mathbf{h}_{i}}$ is an epimorphism, we have $\left( H_{i}\right) _{\alpha }\circ \left( \gamma _{i}\right)
_{a}=\left( \gamma _{i}\right) _{b}\circ \left( M_{i}\right) _{\alpha }$ for
any arrow $\alpha :a\rightarrow b$ in $Q_{1}$. The surjectivity of $\gamma
_{i}$ implies that $\gamma _{i}$ has a right inverse $\gamma _{i}^{-1}=\left(
\left( \gamma _{i}\right) _{a}^{-1}\right) _{a\in Q_{0}}$. Thus,
\begin{equation*}
\left( H_{i}\right) _{\alpha }=\left( \gamma _{i}\right) _{b}\circ \left(
M_{i}\right) _{\alpha }\circ \left( \gamma _{i}\right) _{a}^{-1}.
\end{equation*}%
Hence, the representation $H_{i}=\left( \Bbbk ^{\left( \mathbf{h}_{i}\right)
_{a}},\left( H_{i}\right) _{\alpha }\right) $ is uniquely determined by the epimorphism $\gamma _{i}$.
We complete the proof by showing that
$H_{i}\cong ker\left( \partial _{i-1}\right) /im\left( \partial _{i}\right)$.
The condition $\partial _{i-1}\eta _{i}=0$ means that $im\left( \eta
_{i}\right) \subseteq ker\left( \partial _{i-1}\right) $. On the other hand,
by the isomorphism theorem and the injectivity of $\eta _{i}$, we obtain
$im\left( \eta _{i}\right) \cong \frac{M_{i}}{ker\left( \eta _{i}\right) }\cong M_{i}\cong ker\left( \partial _{i-1}\right) $.
Hence $im\left( \eta_{i}\right) =ker\left( \partial _{i-1}\right) $. Therefore,
\begin{equation*}
H_{i}\cong \frac{M_{i}}{ker\left( \gamma _{i}\right) }=\frac{M_{i}}{im\left(
\phi _{i}\right) }\cong \frac{ker\left( \partial _{i-1}\right) }{im\left(
\partial _{i}\right) }.
\end{equation*}
\end{proof}

Now, we define $comhom_{{\mathbf{d}},{%
\mathbf{r}}}^{\Lambda }$ to be the locally closed subvariety of the affine
space
\begin{eqnarray*}
&&\prod\limits_{i\in \mathbb{Z}}\left( Hom_{\Lambda }\left( P^{\mathbf{d}%
_{i}},P^{\mathbf{d}_{i-1}}\right) \times Hom_{\Lambda }\left( M_{i},P^{%
\mathbf{d}_{i-1}}\right) \times \right.   \notag \\
&&\ \ \ \ \ \left. Hom_{\Lambda }\left( P^{\mathbf{d}_{i}},M_{i}\right)
\times Hom_{\Lambda }\left( M_{i},\Bbbk ^{\mathbf{h}_{i}}\right) \times rep_{%
\mathbf{h}_{i}}^{\Lambda }\right)
\end{eqnarray*}%
consisting of tuples and representations $\left( \partial _{i},\eta
_{i},\phi _{i},\gamma _{i},H_{i}\right) _{i\in \mathbb{Z}},$
\begin{equation*}
\xymatrix{\cdots \ar[r]^{\partial_{i+1}} & P^{\mathbf{d}_i}
\ar[rr]^{\partial_i} \ar[dr]^{\phi_i} & & P^{\mathbf{d}_{i-1}}
\ar[r]^{\partial_{i-1}} & \cdots \\ & & M_i \ar[d]^{\gamma_i}
\ar[ur]^{\eta_i} \\ && H_i},
\end{equation*}%
defined by the conditions $\partial _{i}\partial _{i+1}=0$, $\partial_{i}=\eta_{i}\phi_{i}$, $\partial _{i-1}\eta _{i}=0$, $\gamma _{i}\phi
_{i}=0$, $\eta _{i}$ is a $\Lambda$-monomorphism, $\phi _{i}$ has rank
$\mathbf{r}_{i}$, and $\gamma _{i}$ is an epimorphism. By Lemma %
\ref{comhom}, for any $\left( \partial _{i},\eta _{i},\phi _{i},\gamma
_{i},H_{i}\right) _{i\in \mathbb{Z}}\in comhom_{{\mathbf{d}},{\mathbf{r}}%
}^{\Lambda }$ we have $\left( \partial _{i}\right) _{i\in \mathbb{Z}}\in
comproj_{{\mathbf{d}},{\mathbf{r}}}^{\Lambda }$, $ker(\gamma _{i})=im(\phi
_{i})$, and $H_{i}$ is the unique point in $rep_{\mathbf{h}_{i}}^{\Lambda }$
such that $\gamma _{i}:M_i\rightarrow H_i$ is a $\Lambda $-epimorphism and $H_{i}\cong ker\left(
\partial _{i-1}\right) /im\left( \partial _{i}\right) $.

Let $$rep^\Lambda_{\mathbf{h}}=\prod_{i\in\mathbb{Z}}
rep^\Lambda_{\mathbf{h}_i},$$ which admits an action by $$Gl_\mathbf{h}=\prod_{i\in\mathbb{Z}}Gl_{\mathbf{h}_i}$$
such that orbits are isomorphism classes of $\mathbb{Z}$-graded modules.

The projections $\left( \partial _{i},\eta _{i},\phi _{i},\gamma _{i},H_{i}\right) _{i\in
\mathbb{Z}}\mapsto (\partial_{i})_{i\in \mathbb{Z}}$  and
$\left( \partial _{i},\eta _{i},\phi _{i},\gamma _{i},H_{i}\right) _{i\in
\mathbb{Z}}\mapsto (H_i)_{i\in \mathbb{Z}}$ induces a pair
of morphisms
$$\xymatrix{& comhom_{\mathbf{d},\mathbf{r}}^{\Lambda} \ar[dl]_\pi \ar[dr]^\rho & \\ comproj_{{\mathbf{d}},{\mathbf{r}}}^{\Lambda}
&&rep_{\mathbf{h}}^{\Lambda}}$$
where $\rho(\pi^{-1}(X))$ is the $Gl_\mathbf{h}$-orbit of the homology of
the complex $X$.
Also, $\pi(\rho^{-1}(Y))$ is the set
of complexes $X$, with $H^*X\cong Y$.
The actions lift to an action of $G_\mathbf{d}\times Gl_\mathbf{h}\times Aut_\Lambda M$ on $comhom_{\mathbf{d},\mathbf{r}}^{\Lambda}$, where
$M=\prod_{i\in \mathbb{Z}}M_i$ and $$Aut_\Lambda M=\prod_{i\in \mathbb{Z}}Aut_\Lambda M_i.$$ Moreover, given $Z,Z'\in comhom_{\mathbf{d},\mathbf{r}}^{\Lambda}$,  $Z$ and $Z'$
are conjugate if and only if $\pi(Z)$ and $\pi(Z')$ are conjugate. If $Z$ and $Z'$
are conjugate, then $\rho(Z)$ and $\rho(Z')$ are conjugate, but the converse
is not true in general, as complexes are not determined by their homology in
global dimension two.

The map $\pi$ is surjective. We describe the image of $\rho$. Let $$rep_{\mathbf{h},\mathbf{r}}^{\Lambda }\subseteq
rep_{\mathbf{h}}^{\Lambda }$$ be the subset of representations $(H_i)_{i\in \mathbb{Z}}$
which admit a presentation $$\xymatrix{P^{\mathbf{d}_i} \ar[r]
& M_i \ar[r] & H_i \ar[r] & 0}.$$
The following lemma is well known, but we include it for the sake of
completeness.

\begin{lemma}\label{openrep}
The subset $rep_{\mathbf{h},\mathbf{r}}^{\Lambda }$ is open
in $rep_{\mathbf{h}}^{\Lambda }$.
\end{lemma}

\begin{proof}
It is enough to show that the subset $rep_{\mathbf{h}_{i},\mathbf{r}%
_{i}}^{\Lambda }$ is open in $rep_{\mathbf{h}_{i}}^{\Lambda }$ for a fixed
integer $i$. It is well known that the map $H\mapsto dim\left( \text{Ext}%
^{t}\left( H,S_{a}\right) \right) $ where $S_{a}$ is the simple top of
projective representation $P_{a}$, is upper semi continuous on $rep_{\mathbf{%
h}_{i}}^{\Lambda },$ see \cite{S85} for details. The dimension of Ext$%
^{0}(H,S_{a})$ and Ext$^{1}(H,S_{a})$ are equal to the multiplicities of the
projective $P_{a}$ in the first and second projective, respectively, of a
minimal projective presentation of $H$. Therefore it follows that the set of
representations $H$ admitting a projective presentation $P\rightarrow
M\rightarrow H\rightarrow 0$ with $P$ and $M$ fixed, is an open subvariety
of $rep_{\mathbf{h}_{i}}^{\Lambda }$. Hence, $rep_{\mathbf{h}_{i},\mathbf{r}%
_{i}}^{\Lambda }$ is open in $rep_{\mathbf{h}_{i}}^{\Lambda }$ and the lemma
follows.
\end{proof}

The fact that $im(\rho)$ open in $rep_{\mathbf{h}}^{\Lambda }$ follows from Lemma \ref{openrep} and the following lemma.

\begin{lemma}
$im(\rho)=rep_{\mathbf{h},\mathbf{r}}^{\Lambda }$.
\end{lemma}
\begin{proof}
That $im(\rho)\subseteq rep_{\mathbf{h},\mathbf{r}}^{\Lambda }$
follows by the construction of $comhom_{{\mathbf{d}},{\mathbf{r}}}^{\Lambda}$.
For the converse let, $(H_i)_{i\in \mathbb{Z}}\in
rep_{\mathbf{h},\mathbf{r}}^{\Lambda }$, and let
$$0\rightarrow M_{i+1} \stackrel{\eta_{i+1}} {\longrightarrow}P^{\mathbf{d}_i}
\stackrel{\phi_i}{\longrightarrow} M_i \stackrel{\gamma_i}{\longrightarrow} H_i \rightarrow 0$$
be a projective resolution of $H_i$, for all $i\in \mathbb{Z}$. The tuple
$(\eta_i\phi_i,\phi_i,\eta_i,\gamma_i,H_i)_{i\in \mathbb{Z}}$ maps onto $(H_i)_{i\in \mathbb{Z}}$, and so $im(\rho)\supseteq rep_{\mathbf{h},\mathbf{r}}^{\Lambda }$.
\end{proof}

\begin{remark}
It is possible to define a variety $comhom_{\mathbf{d},\mathbf{r}}^{\Lambda}$, and  maps $\pi$ and $\rho$, for any algebra $\Lambda$, by considering tuples
$\left( \partial _{i},\eta _{i},\phi _{i},\gamma
_{i},M_i,H_{i}\right) _{i\in \mathbb{Z}}$, where $M_i$ is the kernel
of $\partial_{i-1}$.
\end{remark}

\section{Proof of the theorem}

The aim of this section is to prove the main theorem stated in
the introduction. That is, that the two morphisms
$\pi$ and $\rho$, defined in the previous section, are smooth
with irreducible rational fibres.
We do so by showing that they are compositions of isomorphisms, principal
$G$-bundles, open immersions, and vector bundles.

Recall the following
fact from elementary linear algebra, which has a consequence that
right and left inverses of matrices of full rank induce
morphisms of varieties. Let $Mat_{i\times j}$ denote
the vector space of $i\times j$ matrices.

\begin{lemma}
\label{InversMatriks}Let ${U}\in Mat_{i\times j}\left( \Bbbk
\right) $ and ${W}\in Mat_{j\times i}\left( \Bbbk \right) $
with $rank\left( {U}\right) =rank\left( {W}\right) =j\leq i$.
Then, the left inverse of the matrix ${U}$ is
\begin{equation*}
{U}^{-1}=\left( {U}^{T}{U}\right) ^{-1}{U}^{T}=%
\frac{adj\left( {U}^{T}{U}\right) \cdot {U}^{T}}{%
det\left( {U}^{T}{U}\right) }
\end{equation*}%
and the right inverse of the matrix ${W}$ is
\begin{equation*}
{W}^{-1}={W}^{T}\left( {WW}^{T}\right) ^{-1}=\frac{%
{W}^{T}\cdot adj\left( {WW}^{T}\right) }{det\left( {WW}%
^{T}\right) }
\end{equation*}
where ${U}^{T}$ and $adj\left( {U}\right) $ denote the
transpose and adjoint of the matrix ${U}$, respectively.
\end{lemma}

We will need the following result on homomorphism of vector bundles. We
refer to \cite{P97} for details.

\begin{proposition}
\label{VBundle}Let $f:E\rightarrow F$ be a map of vector bundles over $X$.
Suppose that the rank of $f_{x}$ remain constant as $x$ varies over $X$.
Then $ker\left( f\right) $ and $im\left( f\right) $ are sub-bundles of $E$
and $F$, respectively.
\end{proposition}

\subsection{The smooth morphism from $comhom_{\mathbf{d},\mathbf{r}%
}^{\Lambda }$ to $comproj_{\mathbf{d},\mathbf{r}}^{\Lambda }$}

The aim of this subsection is to prove the following lemma, which is the first half
of the theorem.

\begin{lemma}
\label{phi}The morphism
$\pi :comhom_{\mathbf{d},\mathbf{r}}^{\Lambda }\rightarrow comproj_{\mathbf{d}
,\mathbf{r}}^{\Lambda }$
given by
\begin{equation*}
\pi \left( \partial _{i},\eta _{i},\phi _{i},\gamma _{i},H_{i}\right)
_{i\in \mathbb{Z}}=(\partial _{i})_{i\in \mathbb{Z}}
\end{equation*}%
is smooth with irreducible rational fibres.
\end{lemma}

We prove the lemma by decomposing $\pi $ into smooth morphisms
\begin{equation*}
\pi =\pi _{4}\circ \pi _{3}\circ \pi _{2}\circ \pi _{1}
\end{equation*}%
where $\pi _{i}:X_{i-1}\rightarrow X_{i}$ are projection maps and $X_{i}$ are defined as
follows. Let
\begin{equation*}
\begin{array}{rl}
X_{0} & :=comhom_{\mathbf{d},\mathbf{r}}^{\Lambda }, \\
X_{1} & \text{is the set of tuples }\left( \partial _{i},\eta _{i},\phi
_{i},\gamma _{i}\right) _{i\in \mathbb{Z}}\text{ obtained by projecting }%
X_{0}\text{ via }\pi_1,\\
X_{2} & \text{is the set of tuples }\left( \partial _{i},\eta _{i},\phi
_{i}\right) _{i\in \mathbb{Z}}\text{ obtained by projecting }X_{1}\text{ via }\pi_2, \\
X_{3} & \text{is the set of tuples }\left( \partial _{i},\eta _{i}\right)
_{i\in \mathbb{Z}}\text{ obtained by projecting }X_{2}\text{ via }\pi_3, \\
X_{4} & :=comproj_{\mathbf{d},\mathbf{r}}^{\Lambda },%
\end{array}%
\end{equation*}%
and the maps $\pi _{i}$ are restrictions of the projection maps.

\begin{lemma}
\label{PhiSatu}The map $\pi _{1}:comhom_{\mathbf{d},\mathbf{r}}^{\Lambda
}\rightarrow X_{1}$, given by
\begin{equation*}
\pi _{1}(\left( \partial _{i},\eta _{i},\phi _{i},\gamma _{i},H_{i}\right)
_{i\in \mathbb{Z}})=\left( \partial _{i},\eta _{i},\phi _{i},\gamma
_{i}\right) _{i\in \mathbb{Z}}
\end{equation*}%
is an isomorphism of varieties.
\end{lemma}

\begin{proof}
We prove the lemma by constructing an inverse of $\pi _{1}$. After fixing
bases we may view all maps and representations as tuples of matrices. Let
$\left( \partial _{i},\eta _{i},\phi _{i},\gamma _{i}\right) _{i\in
\mathbb{Z}}\in X_{1}$. Since $\gamma _{i}:M_{i}\rightarrow \Bbbk ^{\mathbf{h}%
_{i}}$ is an epimorphism, it has a right inverse $\gamma
_{i}^{-1}=\left( (\gamma _{i})_{a}^{-1}\right) _{a\in Q_{0}}$  by Lemma %
\ref{InversMatriks}. By Lemma \ref{comhom}, we construct a representation
$H_{i}\ $ with $(H_{i})_{\alpha }=(\gamma _{i})_{b}(M_{i})_{\alpha }(\gamma
_{i})_{a}^{-1}$ for any arrow $\alpha:a\rightarrow b$ in $Q_1$. Thus, $H_{i}$ is the unique representation in $rep_{\mathbf{%
h}_{i}}^{\Lambda}$ which makes $\gamma _{i}:M_{i}\rightarrow H_{i}$ is a $\Lambda$-epimorphism. Then, for any $i\in\mathbb{Z}$,
$H_i$ admits a presentation $P^{\mathbf{d}_i}\rightarrow M_i\rightarrow H_i\rightarrow 0$.
Hence, $\left( H_{i}\right) _{i\in \mathbb{Z}}\in rep^{\Lambda}_{\mathbf{h},\mathbf{r}}$. So we have a morphism
\begin{equation*}
\left( \partial _{i},\eta _{i},\phi _{i},\gamma _{i}\right) _{i\in \mathbb{Z}%
}\mapsto \left( \partial _{i},\eta _{i},\phi _{i},\gamma _{i},\left( (\gamma
_{i})_{b}(M_{i})_{\alpha }(\gamma _{i})_{a}^{-1}\right) _{\alpha
:a\rightarrow b}\right) _{i\in \mathbb{Z}}
\end{equation*}%
which is the inverse of $\pi _{1}$. This completes the proof of the lemma.
\end{proof}

The proof of the following lemma is similar to the proof of Lemma 17 in \cite%
{JS05}, and so we skip the details.

\begin{lemma}
\label{PhiDua}The map $\pi _{2}:X_{1}\rightarrow X_{2}$, given by
\begin{equation*}
\pi_2( \left( \partial _{i} , \eta _{i} ,\phi _{i} ,\gamma _{i} \right)
_{i\in \mathbb{Z}})= \left( \partial _{i}, \eta _{i} ,\phi _{i}\right)
_{i\in \mathbb{Z}}
\end{equation*}
is a locally trivial $Gl_{\mathbf{h}}$-bundle.
\end{lemma}

Any locally trivial $Gl_{\mathbf{h}}$-bundle is a smooth morphism, and so $\pi_2$ is
smooth.

\begin{lemma}
\label{PhiTiga}The map $\pi _{3}:X_{2}\rightarrow X_{3}$, given by
\begin{equation*}
\pi_3( \left( \partial _{i} , \eta _{i} , \phi _{i}\right) _{i\in \mathbb{Z}%
})= \left( \partial _{i}, \eta _{i} \right) _{i\in \mathbb{Z}}
\end{equation*}
is an isomorphism of varieties.
\end{lemma}

\begin{proof}
We prove that $\pi _{3}$ is an isomorphism by constructing an inverse
morphism $(\pi _{3})^{-1}$. We fix bases and assume that all maps are given
by matrices. Given any $\left( \partial _{i},\eta _{i}\right) _{i\in \mathbb{%
Z}}\in X_{3}$. Since $\eta _{i}$ is a $\Lambda$-monomorphism, it has a left
inverse $\eta _{i}^{-1}=\left( \left( \eta _{i}\right) _{a}^{-1}\right) _{a\in
Q_{0}}$. For any $i\in \mathbb{Z}$, we construct a $\Lambda $-homomorphism $\phi _{i}=\eta
_{i}^{-1}\partial _{i}$. By Lemma \ref%
{InversMatriks}, the map $(\pi _{3})^{-1}$ defined by
\begin{equation*}
(\pi _{3})^{-1}(\left( \partial _{i},\eta _{i}\right) _{i\in \mathbb{Z}%
})=\left( \partial _{i},\eta _{i},\eta _{i}^{-1}\partial _{i}\right) _{i\in
\mathbb{Z}}
\end{equation*}%
is a morphism of varieties, and it is the inverse of $\pi _{3}$, which
completes the proof of the lemma.
\end{proof}

 Finally, the smoothness of $\pi
_{4}$ follows from the following lemma and the fact that vector bundles and
open immersions are smooth.

\begin{lemma}
\label{PhiEmpat}The map $\pi_{4}:X_{3}\rightarrow comproj_{\mathbf{d},%
\mathbf{r}}^{\Lambda }$, given by
\begin{equation*}
\pi_4(\left( \partial _{i} , \eta _{i} \right) _{i\in \mathbb{Z}})=\left(
\partial _{i} \right) _{i\in \mathbb{Z}}
\end{equation*}
is the composition of an open immersion with a vector bundle with base $%
comproj_{\mathbf{d},\mathbf{r}}^{\Lambda }$.
\end{lemma}

\begin{proof}
Without lost of generality we may write
\begin{equation*}
X_{3}\subseteq comproj_{\mathbf{d},\mathbf{r}}^{\Lambda }\times
\prod\limits_{i\in \mathbb{Z}}Inj_{\Lambda }\left( M_{i},P^{\mathbf{d}%
_{i-1}}\right)
\end{equation*}%
where $\ \ \ Inj_{\Lambda }\left( M_{i},P^{\mathbf{d}_{i-1}}\right) =\left\{ \eta
\in Hom_{\Lambda }\left( M_{i},P^{\mathbf{d}_{i-1}}\right) |\eta \text{ is
injective}\right\} \ \ \ $ and $((\partial _{i})_{i\in \mathbb{Z}},(f_{i})_{i\in
\mathbb{Z}})\in X_{3}$ if $\partial _{i-1}f_{i}=0,$ for all $i\in \mathbb{Z}$%
. By the rank condition, $Inj_{\Lambda }\left( M_{i},P^{\mathbf{d}%
_{i-1}}\right) $ is open in $Hom_{\Lambda }\left( M_{i},P^{\mathbf{d}%
_{i-1}}\right) $. Let$\ $
\begin{equation*}
Y\subseteq comproj_{\mathbf{d},\mathbf{r}}^{\Lambda }\times \prod_{i\in
\mathbb{Z}}Hom_{\Lambda }(M_{i},P^{\mathbf{d}_{i-1}})
\end{equation*}%
consist of pairs $((\partial _{i})_{i\in \mathbb{Z}},(f_{i})_{i\in \mathbb{Z}%
})$ such that $\partial _{i-1}f_{i}=0,$ for all $i\in \mathbb{Z}$. Hence $%
X_{3}$ is open subset of $Y$ and therefore there is an open immersion $%
X_{3}\rightarrow Y$ with image those pairs $((\partial _{i})_{i\in \mathbb{Z}%
},(f_{i})_{i\in \mathbb{Z}})$ with $f_{i}$ injective for all $i\in \mathbb{Z}
$. The projection
\begin{equation*}
\pi ^{\prime }:Y\rightarrow comproj_{\mathbf{d},\mathbf{r}}^{\Lambda }
\end{equation*}%
is the kernel of the morphism of trivial vector bundles
\begin{equation*}
comproj_{\mathbf{d},\mathbf{r}}^{\Lambda }\times \prod_{i\in \mathbb{Z}%
}Hom_{\Lambda }(M_{i},P^{\mathbf{d}_{i-1}})\rightarrow comproj_{\mathbf{d},%
\mathbf{r}}^{\Lambda }\times \prod_{i\in \mathbb{Z}}Hom_{\Lambda }(M_{i},P^{%
\mathbf{d}_{i-2}})
\end{equation*}%
given by
\begin{equation*}
((\partial _{i})_{i\in \mathbb{Z}},(f_{i})_{i\in \mathbb{Z}})\mapsto
((\partial _{i})_{i\in \mathbb{Z}},(\partial _{i-1}f_{i})_{i\in \mathbb{Z}}).
\end{equation*}%
On fibres, the kernel is isomorphic to $\prod_{i\in \mathbb{Z}}Hom_{\Lambda
}(M_{i},ker(\partial _{i-1}))$, which has constant dimension since $M_{i}$
is projective, and so $\pi ^{\prime }$ is a vector bundle by Proposition \ref%
{VBundle}. The map $\pi _{3}$ is therefore the composition of an open
immersion with a vector bundle with base $comproj_{\mathbf{d},\mathbf{r}%
}^{\Lambda }$. The lemma follows.
\end{proof}

As before, let $M=\prod_{i\in \mathbb{Z}}M_i$, and $Aut_\Lambda M =
\prod_{i\in \mathbb{Z}}Aut_\Lambda M_i$.

\begin{corollary}
$\pi_4$ is a locally trivial $Aut_\Lambda M$-bundle.
\end{corollary}

\begin{proof}
The proof of Lemma \ref{PhiEmpat} shows that on fibres, the kernel of the
trivial bundle is isomorphic to $\prod\limits_{i\in \mathbb{Z}%
}Hom_{\Lambda }\left( M_{i},ker\left( \partial _{i-1}\right) \right) $. Since $X_3$ is open in $Y$, on fibres the maps that belong to $X_3$ are the injective maps. Now, the set of all injective maps in $Hom_{\Lambda }\left( M_{i},ker\left( \partial _{i-1}\right) \right)$
is isomorphic to $Aut_{\Lambda
}\left( M_{i}\right)$ since $M_{i}\cong ker\left(
\partial _{i-1}\right)$. Hence, $\pi _{4}$ is locally trivial with fibres $%
Aut_{\Lambda }M$-equivariantly isomorphic to the group $Aut_{\Lambda }M$.
\end{proof}

Having proved that $\pi_i$ are smooth, we can conclude that $\pi$ is smooth.
Also, the previous four lemmas show that the fibres are rational and irreducible, and
so Lemma \ref{phi} follows.

As a consequence of the proofs, we have
the following dimension formula.

\begin{corollary}
\label{dimpi}
$$dim
(comproj_{\mathbf{d},\mathbf{r}}^{\Lambda }) = dim (comhom_{\mathbf{d},\mathbf{r}}^{\Lambda}) - \sum_{i\in \mathbb{Z}}(
\mathbf{h}_i^T\mathbf{h}_i + (\Theta^{-1}\mathbf{k}_i)^T\mathbf{k}_i).$$
\end{corollary}
\begin{proof}
The  sum $\sum_{i\in \mathbb{Z}}
\mathbf{h}_i^T\mathbf{h}_i$ computes the dimension of the fibre of $\pi_2$
and the sum $\sum_{i\in \mathbb{Z}} (\Theta^{-1}\mathbf{k}_i)^T\mathbf{k}_i$
is the fibre dimension of $\pi_4$. The formula follows.
\end{proof}

\subsection{The smooth morphism from $comhom_{\mathbf{d},\mathbf{r}%
}^{\Lambda }$ to $rep_{\mathbf{h}}^{\Lambda }$}

The aim now is to prove the following lemma, which is the second part of the
theorem.

\begin{lemma}
\label{rho}The morphism
$\rho :comhom_{\mathbf{d},\mathbf{r}}^{\Lambda }\rightarrow
rep_{\mathbf{h}}^{\Lambda }
$ given by
\begin{equation*}
\rho ((\partial _{i},\eta _{i},\phi _{i},\gamma _{i},H_{i})_{i\in \mathbb{Z%
}})=H_{i}
\end{equation*}%
is smooth.
\end{lemma}

Similar to the case of $\pi $, we decompose%
\begin{equation*}
\rho =\rho _{4}\circ \rho _{3}\circ \rho _{2}\circ \rho _{1}
\end{equation*}%
and prove that each $\rho _{i}$ is a smooth morphism. We let $\rho
_{i}:Z_{i-1}\rightarrow Z_{i}$ be projection maps, where
\begin{equation*}
\begin{array}{rl}
Z_{0} & :=comhom_{\mathbf{d},\mathbf{r}}^{\Lambda }, \\
Z_{1} & \text{is the set of tuples }(\eta _{i},\phi _{i},\gamma
_{i},H_{i})_{i\in \mathbb{Z}}\text{ obtained by projecting }Z_{0}\text{ via }\rho_1,
\\
Z_{2} & \text{is the set of tuples }(\phi _{i},\gamma
_{i},H_{i})_{i\in \mathbb{Z}}\text{ obtained by projecting }Z_{1}\text{ via }\rho_2, \\
Z_{3} & \text{is the set of tuples }(\gamma
_{i},H_{i})_{i\in \mathbb{Z}}\text{ obtained by projecting }Z_{2}\text{ via }\rho_3, \\
Z_{4} & :=rep_{\mathbf{h},\mathbf{r}}^{\Lambda },%
\end{array}%
\end{equation*}
and the maps $\rho _{i}$ are restrictions of the projection maps.

\begin{lemma}
\label{RhoSatu}The map
$\rho _{1}:comhom_{\mathbf{d},\mathbf{r}}^{\Lambda }\rightarrow Z_{1}
$ given by
\begin{equation*}
\rho _{1}(\left( \partial _{i},\eta _{i},\phi _{i},\gamma
_{i},H_{i}\right) _{i\in \mathbb{Z}}=\left( \eta _{i},\phi _{i},\gamma
_{i},H_{i}\right) _{i\in \mathbb{Z}}
\end{equation*}%
is an isomorphism of varieties.
\end{lemma}

\begin{proof}
The morphism $\rho_{1}$ has an inverse given by
\begin{equation*}
\left( \eta _{i},\phi _{i},\gamma _{i},H_{i}\right) _{i\in \mathbb{Z}%
}\mapsto \left( \eta _{i}\phi _{i},\eta _{i},\phi _{i},\gamma
_{i},H_{i}\right) _{i\in \mathbb{Z}},
\end{equation*}%
due to the equation $\partial _{i}=\eta _{i}\phi _{i}$ in the definition
of $comhom_{\mathbf{d},\mathbf{r}}^{\Lambda }$.
\end{proof}

Since open immersions and
vector bundles are smooth, the following lemma proves smoothness of $\rho
_{2}$.

\begin{lemma}
\label{RhoDua}The map $\rho _{2}:Z_{1}\rightarrow Z_{2}$ defined by
\begin{equation*}
\rho_2(\left( \eta _{i}, \phi _{i} , \gamma _{i}, H_{i} \right) _{i\in
\mathbb{Z}})= \left( \phi _{i} , \gamma _{i} , H_{i}\right) _{i\in \mathbb{Z}%
}
\end{equation*}
is the composition of an open immersion with a vector bundle with base $Z_2$.
\end{lemma}

\begin{proof}
Let $Y_{1}$ be defined as $Z_{1}$, but without the restriction that $\eta
_{i}$ should be a monomorphism.
Then $Y_1$ is isomorphic to the kernel of the homomorphism of trivial vector bundles
$$\prod_{i\in \mathbb{Z}}Hom(M_i,P^{\mathbf{d}_{i-1}})\times Z_2 \rightarrow \prod_{i\in\mathbb{Z}}Hom(M_i,M_{i-1})\times Z_2$$
given by
$$(\eta_i,\phi_i,\gamma_i,H_i)_{i\in\mathbb{Z}}\mapsto
(\phi_{i-1}\eta_i,\phi_i,\gamma_i,H_i)_{i\in\mathbb{Z}}.$$
By Lemma \ref{comhom}, $im(\eta_i)=ker(\partial_{i-1})=ker(\phi_{i-1})$ so that $\phi_{i-1}\eta_i=0$.
On fibres, the kernel is isomorphic to $\prod_{i\in \mathbb{Z}}Hom_\Lambda(M_i,ker(\partial_{i-1}))$,
which have constant dimension since $M_i$ is projective. Hence, $Y_1\rightarrow Z_2$ is a vector bundle by Proposition \ref{VBundle}.
There is an open
immersion $Z_{1}\rightarrow Y_{1}$ with image those tuples with $\eta _{i}$
injective for all $i$. Therefore $\rho _{2}$ is the composition of an open
immersion with a vector bundle with base $Z_{2}$, and the proof is complete.
\end{proof}

Although, the fibres of $\rho_2$ are isomorphic to
$Aut_\Lambda M$, they are not in general closed under the action of
$Aut_\Lambda M$, and so $\rho_2$ is
not a $Aut_\Lambda M$-bundle. This is because homology of a complex
does not determine its quasi-isomorphism class for an algebra
of global dimension two.

\begin{lemma}
\label{RhoTiga}The map $\rho _{3}:Z_{2}\rightarrow Z_{3}$ defined by
\begin{equation*}
\rho _{3}(\left( \phi _{i},\gamma _{i},H_{i}\right) _{i\in \mathbb{Z}%
})=\left( \gamma _{i},H_{i}\right) _{i\in \mathbb{Z}}
\end{equation*}%
is the composition of an open immersion with a vector bundle with base $%
Z_{3} $.
\end{lemma}

\begin{proof}
Let $Y_{2}$ be defined as $Z_{2}$ by changing the property $im\left( \phi
_{i}\right) =ker\left( \gamma _{i}\right) $ into $\gamma _{i}\phi _{i}=0$
for all $i\in \mathbb{Z}.$  Similar to the proof of Lemma \ref{PhiEmpat},
the projection $Y_{2}\rightarrow Z_{3}$ is the kernel of a homomorphism
between two trivial vector bundles with base $Z_{3}$, which has fibres
isomorphic to $\prod_{i\in \mathbb{Z}}Hom_{\Lambda }(P^{\mathbf{d}%
_{i}},ker(\gamma _{i}))$ and so $Y_{2}\rightarrow Z_{3}$ is a vector bundle.
There is an open immersion $Z_{2}\rightarrow Y_{2}$ with image those tuples
with $im\left( \phi _{i}\right) =ker\left( \gamma _{i}\right) $ for all $i$.
Therefore $\rho _{3}$ is the composition of an open immersion with a vector
bundle with base $Z_{3}$, and the proof is complete.
\end{proof}

\begin{lemma}
\label{RhoEmpat}The map $\rho _{4}:Z_{3}\rightarrow
rep_{\mathbf{h},\mathbf{r}}^{\Lambda }$ defined by
\begin{equation*}
\rho _{4}(\left( \gamma _{i},H_{i}\right) _{i\in \mathbb{Z}})=\left(
H_{i}\right) _{i\in \mathbb{Z}}
\end{equation*}%
is the composition of an open immersion with a vector bundle with base $rep_{\mathbf{h},\mathbf{r}}^{\Lambda }$.
\end{lemma}

\begin{proof}
Let $Y_{3}$ be defined as $Z_{3}$, but without the restriction that $\gamma
_{i}$ is an epimorphism.
We prove that the morphism $Y_{3}\rightarrow
rep_{\mathbf{h},\mathbf{r}}^{\Lambda }$ is the kernel
of a map between trivial vector bundles.
Choose $i\in \mathbb{Z}$ and for simplicity write $(\gamma,H)=(\gamma_i,H_i)$.
There is a homomorphism of vector bundles
\begin{equation*}
Hom\left( M_i,\Bbbk ^{\mathbf{h}_i}\right) \times rep_{\mathbf{h}_i,\mathbf{r}_i}^{\Lambda } \rightarrow
\prod\limits_{\alpha :a\rightarrow b}Hom_{\Bbbk }\left( \Bbbk ^{(k_i)_{a}},\Bbbk
^{(h_i)_{b}}\right)\times rep_{\mathbf{h}_i,\mathbf{r}_i}^{\Lambda }
\end{equation*}%
\begin{equation*}
\left( \left( \gamma _{a}\right) ,\left( H_{\alpha }\right)
\right) \mapsto \left( \left( H_{\alpha }\circ \gamma
_{a}-\gamma _{b}\circ (M_i)_{\alpha }\right) ,\left( H_{\alpha }\right) \right) .
\end{equation*}%
On fibres, the kernel is isomorphic to $Hom_{\Lambda }\left(
M_i,H\right) $ which has constant dimension since $M_i$ is
projective, and so the kernel is a sub bundle by Proposition \ref{VBundle}.
Since $Z_3\subseteq Y_3$ is open it follows that $\rho_4$
is the composition of an open immersion with a vector bundle.
\end{proof}

Lemma \ref{rho} follows from the previous four lemmas. We have the following formula relating the
dimension of $comproj_{\mathbf{d},\mathbf{r}}^{\Lambda }$
to the dimension of $rep_{\mathbf{h},\mathbf{r}}^\Lambda$.
This concludes the proof of the theorem.

\begin{corollary}
$$dim (comproj_{\mathbf{d},\mathbf{r}}^{\Lambda }) =
dim (rep_{\mathbf{h},\mathbf{r}}^\Lambda) + \sum_{i\in \mathbb{Z}}(
\mathbf{d}_i^T\mathbf{r}_i - (\Theta^{-1}\mathbf{k}_i)^T\mathbf{k}_i).$$
\end{corollary}
\begin{proof}
The  sum $\sum_{i\in \mathbb{Z}}
(\mathbf{h}_i^T\mathbf{h}_i+\mathbf{d}_i^T\mathbf{r}_i)$ is the dimension
of the fibre of $\rho$. The corollary now follows from Corollary \ref{dimpi}.
\end{proof}

\section{Applications and Examples}

We start by proving the two corollaries stated in the introduction.

\begin{proof}[Proof of Corollary 2.]
Each of the morphisms $\pi_i:X_{i-1}\rightarrow X_i$ induce
a bijection between the irreducible components of $X_{i-1}$ and
its image $im(\pi_i)=X_i$.
As $\pi$ is surjective, we have a bijection between the irreducible
components of $comproj^\Lambda_{\mathbf{d},\mathbf{r}}$
and $comhom^\Lambda_{\mathbf{d},\mathbf{r}}$. Similarly, there
is a bijection for $comhom^\Lambda_{\mathbf{d},\mathbf{r}}$
and the image of $\rho$. The corollary follows.
\end{proof}

\begin{proof}[Proof of Corollary 3.]
Since $\Lambda$ is hereditary, $rep_{\mathbf{h}%
}^{\Lambda }$ is vector space which is both smooth and rational. Using the
morphisms $\rho_i$ and $\pi_i$ we therefore have that $comproj_{\mathbf{d},\mathbf{r}}^{\Lambda }$ is both smooth and rational. The irreducibility
of $comproj_{\mathbf{d},\mathbf{%
r}}^{\Lambda }$ follows from Corollary 2.
\end{proof}

We end this paper with an example.

\begin{example}
Let $\Lambda$ be given the quiver
\begin{equation*}
\xymatrix{1 \ar[r]^\alpha & 2 \ar[r]^\beta & 3}
\end{equation*}%
with the relation $\beta \alpha =0.$ Consider the dimension array
\begin{equation*}
\mathbf{d}=(\mathbf{d}_{0},\mathbf{d}_{1},\mathbf{d}_{2})=\left( \left(
\begin{array}{c}
2 \\
2 \\
2%
\end{array}%
\right) ,\left(
\begin{array}{c}
2 \\
4 \\
1%
\end{array}%
\right) ,\left(
\begin{array}{c}
2 \\
3 \\
2%
\end{array}%
\right) \right)
\end{equation*}%
Then $comproj_{\mathbf{d}}^{\Lambda }$ consist of matrices $x$ and $y$,%
\begin{equation*}
0\rightarrow P_{1}^{2}\oplus P_{2}^{3}\oplus P_{3}^{2}\xrightarrow{x}%
P_{1}^{2}\oplus P_{2}^{4}\oplus P_{3}^{1}\xrightarrow{y}P_{1}^{2}\oplus
P_{2}^{2}\oplus P_{3}^{2}\rightarrow 0
\end{equation*}%
where%
\begin{equation*}
x=\left(
\begin{array}{ccc}
x_{1} & x_{2} & 0 \\
0 & x_{3} & x_{4} \\
0 & 0 & x_{5}%
\end{array}%
\right) \text{ and }y=\left(
\begin{array}{ccc}
y_{1} & y_{2} & 0 \\
0 & y_{3} & y_{4} \\
0 & 0 & y_{5}%
\end{array}%
\right)
\end{equation*}%
and each $x_{i}$, $y_{i}$ is a matrix block, e.g.%
\begin{equation*}
x_{2}=\left(
\begin{array}{cc}
x_{2,11} & x_{2,12} \\
x_{2,21} & x_{2,22}%
\end{array}%
\right) .
\end{equation*}%
The defining equations of $comproj_{\mathbf{d}}^{\Lambda }$ are obtained
from the matrix equations%
\begin{equation*}
yx=0\text{ and }y_{2}x_{5}=0,
\end{equation*}%
the first giving a complex and the second coming from the relation $\beta
\alpha =0$. Let
\begin{equation*}
\mathbf{r}=(\mathbf{r}_{1},\mathbf{r}_{2})=\left( \left(
\begin{array}{c}
0 \\
2 \\
1%
\end{array}%
\right) ,\left(
\begin{array}{c}
0 \\
1 \\
1%
\end{array}%
\right) \right)
\end{equation*}%
denote the ranks of the matrices, i.e. the dimension vector of the images of
$y$ and $x$. We compute and find%
\begin{equation*}
M_{0}=M_{1}=M_{2}=P^{\mathbf{d}_{0}}.
\end{equation*}%
The dimension vectors of the corresponding homology modules are%
\begin{equation*}
\mathbf{h}=(\mathbf{h}_{0},\mathbf{h}_{1},\mathbf{h}_{2})=\left( \left(
\begin{array}{c}
2 \\
2 \\
3%
\end{array}%
\right) ,\left(
\begin{array}{c}
2 \\
3 \\
3%
\end{array}%
\right) ,\left(
\begin{array}{c}
2 \\
4 \\
4%
\end{array}%
\right) \right) .
\end{equation*}%
Now $rep_{\mathbf{h}_{0}}^{\Lambda }$ has three irreducible components,
given as the orbit closures of the modules%
\begin{equation*}
P_{1}^{2}\oplus P_{3}^{3},~~~S_{1}^{2}\oplus P_{2}^{2}~~~\text{and~~~}%
P_{1}\oplus P_{2}\oplus P_{3}^{2}\oplus S_{1},
\end{equation*}%
where $S_{i}$ is the simple representation at vertex $i$. Of these, only the
latter two have the required presentation%
\begin{equation*}
P^{\mathbf{d}_{1}}\rightarrow M_{0}\rightarrow H_{0}\rightarrow 0.
\end{equation*}%
The variety $rep_{\mathbf{h}_{1}}^{\Lambda }$ also has three irreducible
components, given by%
\begin{equation*}
S_{1}^{2}\oplus P_{2}^{3},~~~P_{1}^{2}\oplus P_{2}\oplus P_{3}^{2}~~~\text{%
and}~~\ S_{1}\oplus P_{1}\oplus P_{2}^{2}\oplus P_{3},
\end{equation*}%
where only the latter two have the required presentation. In total there are
therefore 4 irreducible components in $rep_{\mathbf{h},\mathbf{r}}^{\Lambda }
$, and therefore also in $comproj_{\mathbf{d},\mathbf{r}}^{\Lambda }.$
\end{example}

\vspace{1.5cc}
\noindent{\bf Acknowledgement.}
The authors would like to thank the referee for many valuable suggestions.

%=====================================

\vspace{1cc}

\noindent \textsc{Darmajid} : Algebra Research Division, Institut Teknologi Bandung.\\
Jalan Ganeca no.10, Bandung, Indonesia.\\
Email : darmajid@students.itb.ac.id.\\
\newline
\noindent \textsc{Bernt Tore Jensen} : Gj\o vik University College \\
Teknologivn 22, 2815, Gj\o vik, Norway.\\
Email : bernt.jensen@hig.no

\begin{thebibliography}{9}
\bibitem{ASS06} \textsc{I. Assem, D. Simson, A. Skowronski}, \emph{Elements
of The Representation Theory of Associative Algebras, in : Techniques of
Representation Theory}, Vol. 1, Cambridge University Press, New York, (2006).

\bibitem{JSZ05} \textsc{B. T. Jensen, X. Su, A. Zimmermann}, Degenerations
for derived categories, \emph{J. Pure Appl. Algebra} \textbf{198}, no. 1-3,
281--295, (2005).

\bibitem{JS05} \textsc{B. T. Jensen, X. Su}, Singularities in derived
categories, \emph{Manuscripta Mathematica}, \textbf{117}, no. 4, 475--490,
(2005).

\bibitem{P97} \textsc{J. L. Potier}, \emph{Lectures on vector bundles},
Translated by A. Maciocia, Cambridge Studies in Advanced Mathematics,
\textbf{54}, Cambridge University Press, Cambridge, (1997), viii+251 pp.
ISBN: 0-521-48182-1.

\bibitem{R86} \textsc{C. Riedtmann}, Degenerations for representations of
quivers with relations, \emph{Annales Scientifiques de l}'\'{E}\emph{cole
Normale Sup\'{e}rieure,} \textbf{4}, 275--301, (1986).
\bibitem{S85} \textsc{A. Schofield}, Bounding the global dimension in terms of the dimension,
\emph{Bull. London Math. Soc.}, \textbf{17},
393--394, (1985).
\bibitem{SH-Z01} \textsc{M. Saorin, B. Huisgen-Zimmermann}, Geometry of
chain complexes and outer automorphisms under derived equivalence, \emph{%
Transactions of the American Mathematical Society}, \textbf{353},
4757--4777, (2001).
\end{thebibliography}
\end{document}